\documentclass[12pt,a4paper]{article}
\usepackage[a4paper,
            bindingoffset=0.2in,
            left=1in,
            right=1in,
            top=1in,
            bottom=1in,
            footskip=.25in]{geometry}
\usepackage{amsmath}
\usepackage{amsfonts}
\usepackage{graphicx}
\usepackage{cite}
\usepackage{amsthm}
\usepackage{booktabs}
\usepackage{epstopdf}
\usepackage{float}
\usepackage{authblk}
\usepackage{amsmath,amsthm,amsfonts,amssymb,amsbsy,latexsym,graphicx}
\usepackage{multirow}%
\usepackage{subcaption}
\usepackage{cite}
\usepackage{mathrsfs}%
\usepackage[title]{appendix}%
\usepackage{xcolor}%
\usepackage{textcomp}%
\usepackage{manyfoot}%
\usepackage{booktabs}%
\usepackage{algorithm}%
\usepackage{algorithmicx}%
\usepackage{algpseudocode}%
\usepackage{listings}%
\usepackage{booktabs}
\algdef{SE}[DOWHILE]{Do}{doWhile}{\algorithmicdo}[1]{\algorithmicwhile\ #1}%

\newtheorem{lemma}{Lemma}
\newtheorem{theorem}{Theorem}
\newtheorem*{theorem*}{Theorem}
\newtheorem{example}{Example}

\newtheorem{remark}{Remark}%

\title{Optimal 3D Road Alignment on Topographic Surfaces: A Convergent Dynamic Programming Approach}
\author[1]{M.~E. Abbasov\thanks{m.abbasov@spbu.ru, abbasov.majid@gmail.com}}
\author[1]{A.~A. Gorbunova\thanks{d-anyutik@mail.ru}}
\affil[1]{St. Petersburg State University, SPbSU, 7/9 Universitetskaya nab., St. Petersburg, 199034 Russia}

\begin{document}

\maketitle

\begin{abstract}We consider the problem of finding an optimal 3D road trajectory between two points on a terrain with variable elevation. Unlike common heuristic pathfinding methods, we propose a rigorous framework based on the calculus of variations, introducing an integral cost functional that incorporates material delivery and construction expenses. The existence of a global minimizer is established via the Arzel\`{a}--Ascoli theorem. To solve the problem numerically, we develop a dynamic programming scheme and provide a formal convergence proof. We prove that the sequence of piecewise-linear solutions converges to the true optimum when the grid discretization steps follow a specific power-law relation -- specifically, when the vertical step size decays faster than the horizontal one. To enhance efficiency, we introduce a local-search modification that reduces computational complexity to nearly quadratic \(O(\tau^{-2-\varepsilon})\), where \(\tau\) is the discretization step along the \(x\)-axis. Numerical experiments on 2D and 3D terrains validate the theoretical results, showing that our approach achieves accuracy comparable to the Ritz method while significantly reducing processing time.
\end{abstract}

\subsection*{Keywords}
3D road alignment; dynamic programming; calculus of variations; convergence analysis; path planning; terrain modeling; computational complexity; mathematical modelling; optimization

\subsection*{MSC} 49L20, 65K10, 90C39

\section{Introduction}\label{sec1}

The problem of constructing a road connecting two given points on the terrain naturally arises in civil engineering when designing infrastructure projects related to the laying of roads and railways, pipelines, water conduits, canals, power transmission networks \cite{AME_pp_appl,AME_Won_2015,AME_Wu_2005,AME_Kim_2016,AME_Zwickl-Bernhard_2023}. Such problems are usually solved by heuristic methods that involve the construction of a cost grid, i.e. the area under consideration is divided by a uniform grid into square subareas, each of which is assigned a certain price based on expert opinion or other assumptions \cite{AME_bib_1,AME_bib_2,AME_bib_3}. Thus, when constructing a route, we must move from the cell containing the start point to the cell corresponding to the end point, so that the transition from one cell to another is made exclusively along adjacent cells. This approach reduces the original problem to finding the minimum path on a graph connecting two given nodes. To solve the resulting problem, Dijkstra's algorithm \cite{AME_Dikstra,AME_opt_eng_1} is widely used. To obtain a more accurate solution, it is necessary to increase the density of the grid used, which dramatically increases the computational cost. To overcome these difficulties, researchers use heuristic modifications of Dijkstra's algorithm, such as the A* \cite{AME_A_star,AME_A_star2, AME_A_star3} or D* \cite{AME_opt_eng_2,AME_pp_appl2,AME_Dakulovic_2011} search algorithms. An alternative heuristic approach is based on the utilization of rapidly expanding random trees (RRT) \cite{AME_opt_RRT_ini, AME_opt_RRT_ini2, AME_opt_RRTstar, AME_RRT_connnect, AME_RRT_transit, AME_RRT_so-on, AME_RRT_so-on2,AME_Bruce_2003}.

The aforementioned approaches are not capable of finding a solution with the precision required. In order to address this issue, in paper \cite{AME_Abbasov_2021} a mathematical formalization of the problem was proposed, reducing it to a calculus of variations problem with fixed ends. The authors developed a cost functional, the minimum of which is achieved on the desired optimal trajectory. This novel approach enabled the employment of a vast range of theoretically well-founded methods to address the original problem. Specifically, the necessary minimum condition for the functional, expressed as an integro-differential equation, was obtained. Subsequent analysis invoking the Schauder fixed-point theorem established the existence of a solution to the equation \cite{AME_Abbasov_2024}.  Furthermore, a finite-difference numerical method was developed for solving the resulting boundary value problem. This method ensures the attainment of solutions with any desired accuracy. A  pivotal point of the proposed model was the assumption of negligible terrain elevation differences, thereby reducing the problem to a planar one. This, in turn, led to a substantial simplification of the calculations. In this paper, we eliminate this assumption by considering the trajectory in three-dimensional space. The model obtained through this approach is analyzed, and the existence of an optimum is demonstrated. 

Also in the current research proposes an algorithm for finding a global solution using the dynamic programming method, and its convergence is proven. Additionally, a modification of the method is built. In general, it finds only local solutions, but it significantly reduces computational costs. A comparison of the computational complexity of these two algorithms was carried out. The developed algorithms can be transferred to problems with constraints almost without changes.

It must be emphasized that, due to the specific shape of the graph employed in the problem, the computational complexity of Dijkstra's algorithm and the initially considered dynamic programming method are comparable. However, the latter allowed us to develop a more computationally efficient modification based on it, as well as to derive the convergence conditions for the sequence of solutions obtained for increasingly finer grids. This result enables the determination of the appropriate law for increasing grid density to achieve a solution with any predetermined level of accuracy. Thus, in contrast to the commonly employed Dijkstra's algorithm in this area, the approach developed and explored in this work facilitates the relatively computationally simple attainment of a solution to the original problem with any desired accuracy.

The paper is organized as follows: the problem description and formalization are given in Sect. 2. The dynamic programming scheme is presented in Sect. 3.  In Sect. 4 the existence of the solution and the convergence of the algorithm are proved. The local search modification is discussed in Sect. 5. Numerical experiments are provided in Sect. 6.

\section{Problem statement and formalization}\label{AM_AG_sec2}

Consider the problem of finding the trajectory of the road with the minimum construction cost connecting two given points on the relief of a terrain. Let us introduce the Cartesian coordinate system $Oxyz$. Without loss of generality, we can assume that the first point coincides with the origin. Denote the coordinates of the second point by $(l,y_l,z_l)$. Consider a function $z=\varphi(x,y)$ that defines the relief of the area under consideration. We assume that $\varphi(x,y)$ has continuous partial derivatives up to the second order. Since the endpoints are on the terrain, we have
$$\varphi(0,0)=0,\quad z_l=\varphi(l,y_l).$$

Obviously, the total cost of road construction depends on two factors:
\begin{itemize}
\item The cost of delivering the materials to the construction site, 
\item Wages of workers, cost of construction materials, as well as rental of construction equipment, etc.
\end{itemize}

Suppose that building materials are delivered from the starting point along the already constructed section of road. Introduce two functions.
\begin{itemize}
\item $\alpha\colon\mathbb{R}^2\to\mathbb{R}$ is a delivery cost of the materials needed to build a unit length of the road at the given point per unit length of the road. In other words, the delivery of materials needed to build one meter of road at point $(x,y)$ costs us $\alpha(x,y)$ per one meter of road. We assume that $\alpha(x,y)$ has continuous partial derivatives.
\item $\beta\colon\mathbb{R}^2\to\mathbb{R}$ is a salary of workers, cost of construction materials, as well as rent of construction equipment, etc., required to build the unit length of the roadway at the given point. We assume that $\beta(x,y)$ has continuous partial derivatives.
\end{itemize}

It should be noted that the trajectory is defined by function $y\colon\mathbb{R}^2\to\mathbb{R}$ as the altitude of the point $(x,y(x))$ can be found using $\varphi$ as $z=\varphi(x,y(x))$. 

The total cost of constructing the entire trajectory $y(x)$ is given by
\begin{equation}\label{AM_AG_J}
\begin{split}
J(y)=\int_{0}^{l}\alpha(x,y(x))\sqrt{1+{y'}^2(x)+{z'}^2(x)}\int_{0}^{x} \sqrt{1+{y'}^2(\xi)+{z'}^2(\xi)}d\xi dx\\
+\int_{0}^{l}\beta(x,y(x))\sqrt{1+{y'}^2(x)+{z'}^2(x)}dx,
\end{split}
\end{equation}
where $z(x)=\varphi(x,y(x))$ and ${z'}(x)=\varphi_x(x,y(x))+\varphi_y(x,y(x)){y'}(x)$.

Hence, we get the problem of finding the function $y(x)$ which minimizes the functional (\ref{AM_AG_J}) and satisfies boundary conditions 
\begin{equation}\label{AM_AG_init_cond}
y(0)=0,\quad y(l)=y_l. 
\end{equation}

\section{Dynamic programming scheme}\label{AM_AG_sec3}

It is clear from the problem statement that any part of the optimal trajectory is optimal. In other words, it doesn't matter how we get to a current point when we compute the subsequent parts of the optimal trajectory. This means that Bellman's principle of optimality is valid for our problem and therefore we can apply the dynamic programming scheme \cite{AME_ Bellman_1957,AME_ Moiseev_1971,AME_Sniedovich_2010}.

Let us introduce $\tau=l/n$ for some natural $n$ and divide interval $[0,l]$ by $n$ equidistant nodes $x_0=0$, $x_1=1/n$, $\dots$, $x_i=i\tau$, $\dots$, $x_n=l$. Denote the corresponding coordinates of the trajectory by $y_i$ and $z_i$, where $i=0,\dots,n$. Since the endpoints are fixed and the trajectory belongs to the graph of the function $\varphi$, we have 
$y_0=0$, $y_n=y_l$ and $z_i=\varphi(x_i,y_i)$ for all $i=0,\dots,n$. Thus, a piecewise linear approximation to the solution is defined by $y_i$, $i=1,\dots,n-1$. To formalize the problem, we can assume that the distance between the points $(x_i,y_i)$ and $(x_{i+1},y_{i+1})$ is defined by the functional (\ref{AM_AG_J}) is therefore given by the function
\begin{equation}\label{AM_AG_J_i}
\begin{split}
J_i(y_i,y_{i+1})=\int_{x_i}^{x_{i+1}}\alpha(x,y_i(x))\sqrt{1+{y'_i}^2+{z'_i}^2(x)}\int_{0}^{x} \sqrt{1+{y'_i}^2+{z'_i}^2(\xi)}d\xi dx\\
+\int_{x_i}^{x_{i+1}}\beta(x,y_i(x))\sqrt{1+{y'_i}^2+{z'_i}^2(x)}dx,
\end{split}
\end{equation}
where 
$$y_i(x)=y_i+\frac{y_{i+1}-y_i}{\tau}(x-x_i),$$
$$y'_i=\frac{y_{i+1}-y_i}{\tau},$$
$${z'_i}(x)=\varphi_x(x,y_i(x))+\varphi_y(x,y_i(x))y'.$$
So we need to find a polygonal chain of minimum length that connects two endpoints in a feasible region of the plane $Oxy$, or what is the same, we need to find the solution of the following additive nonlinear programming problem
$$\begin{cases}
\min &J(y_0,\dots,y_n)\\
\textit{s.t. } &y_i\in G_i,\ i=0,\dots,n,
\end{cases}$$
where 
$$J(y_0,\dots,y_n)=\sum_{i=0}^{n-1} J_i(y_i,y_{i+1}),$$ and
$G_i$ are some feasible sets. Note that $G_0=\{0\}$ and $G_n=\{y_l\}$ in our fixed ends problem, but they are not necessarily singletons in the general case. 

Denote 
$$\Pi_i=\left\{(x_i,y)\mid y\in G_i\right\},\ i=0,\dots,n,$$ and assume that we can compute the minimum distance $d(y_i)$ between any point $(x_{i},y_{i})$, $y_i\in G_i$ and $\Pi_0$. 
Choose an arbitrary $y_{i+1}\in G_{i+1}$ and introduce the distance between the point $(x_{i+1},y_{i+1})$ and set $\Pi_0$ as follows:
\begin{equation}\label{AM_AG_d_i}
d(y_{i+1})=\min_{y_i\in G_i}\left(d(y_i)+J_i(y_i,y_{i+1})\right),
\end{equation} 
where $d(y_0)=0$.
Since 
$$\min_{y_{0}\in G_{0},\dots,y_{i}\in G_{i}}J(y_0,\dots,y_n)=d(y_{i+1})+\sum_{j=i+1}^{n-1}J_j(y_j,y_{j+1}),$$
the part of optimal polygonal chain connecting $(x_{i+1},y_{i+1})$ with 
$\Pi_0$ must have a length equal to $d(y_{i+1})$. Any polygonal chain that does not satisfy this property cannot be the solution of the problem and should be discarded. Moving in this way from $\Pi_0$ to $\Pi_n$ at the $n$-th step, we define
\begin{equation}\label{AM_AG_d_n}
\min_{y_n\in G_n}d(y_n)
\end{equation}
and find the solution by choosing the polygonal chain on which the minimum (\ref{AM_AG_d_n}) is attained.

To build a numerical algorithm using this scheme, we need to discretize the problem by the parameter $y$. So we define a grid in the plane $Oxy$ with step $\tau$ with respect to the variable $x$ and step $\Delta$ with respect to $y$. Label the nodes of the grid as $(x_i,y^k_i)$, where $k$ is the number of the node in the set $\Pi_i$. The distance between two nodes $(x_i,y^k_i)$ and $(x_{i+1},y^j_{i+1})$, which are on neighboring sets $\Pi_i$ and $\Pi_{i+1}$ is defined as
\begin{equation}\label{AM_AG_d_kj}
d_i^{kj}=J_i(y^k_i,y^j_{i+1}).
\end{equation} 
Thus, our problem is reformulated as follows: among all the broken lines connecting $\Pi_0$ with $\Pi_n$ and having vertices in the nodes of the grid, we need to find the one with the minimum length.

Denoting the length of the broken line of the minimum length connecting $(x_i,y^k_i)$ with $\Pi_0$ by $d_i^k$, we can rewrite (\ref{AM_AG_d_i}) in the form
\begin{equation}\label{AM_AG_d_is}
d^s_{i+1}=\min_{k}\left(d^k_i+d_i^{ks}\right).
\end{equation} 
The minimum is taken over all those numbers $k$ for which the corresponding nodes belong to $\Pi_k$. After the last iteration we select the piecewise linear trajectory of the length 
\begin{equation}\label{AM_AG_d_n_discrete}
\min_{k} d_n^k,
\end{equation} 
which is obviously the solution of the problem. Note that for the fixed endpoints $\Pi_0$ and $\Pi_n$ are singletons and therefore step (\ref{AM_AG_d_n_discrete}) can be skipped.

\begin{figure}[h]
\centering
\includegraphics[width=0.75\textwidth]{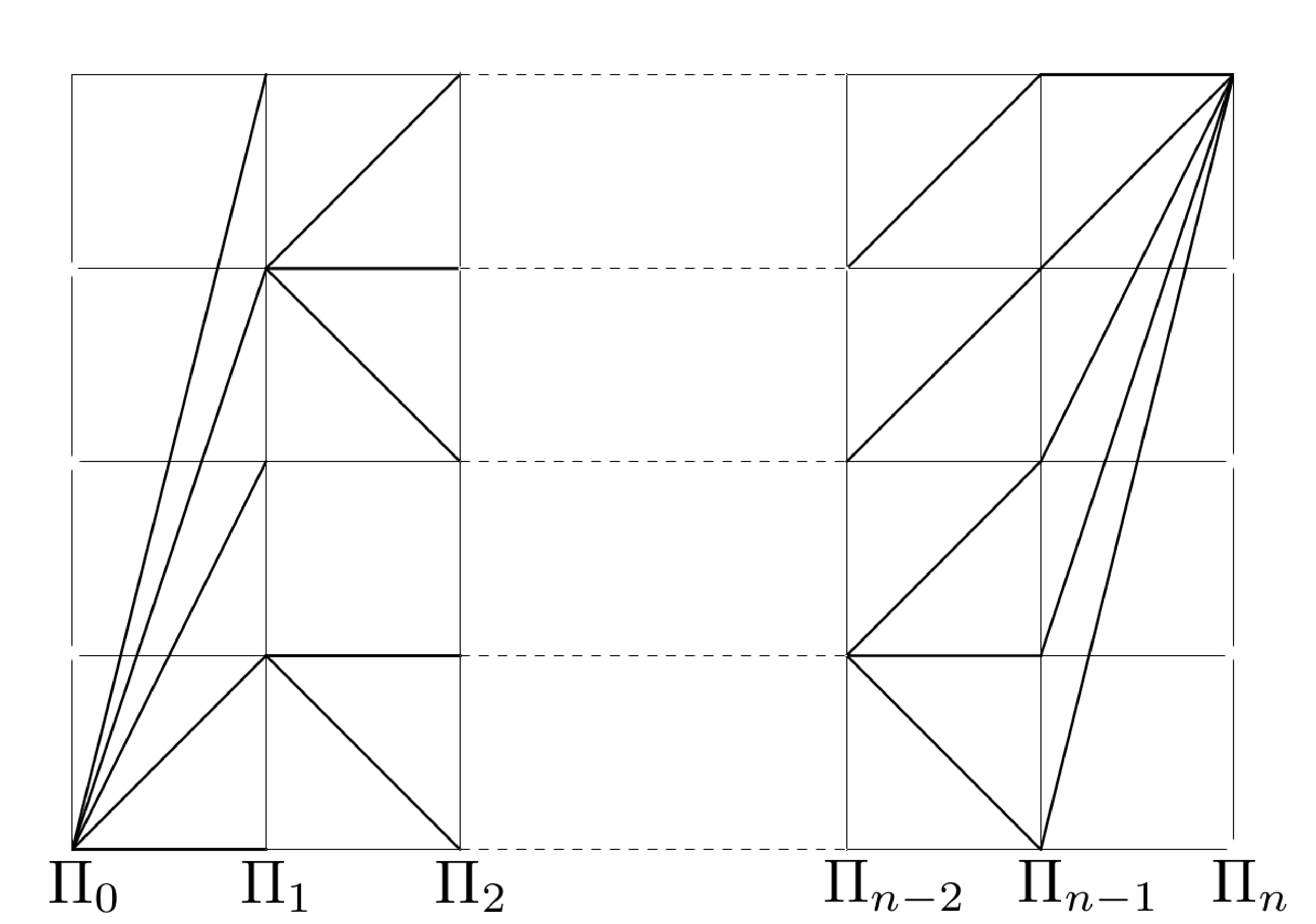}
\caption{An illustration of the dynamic programming algorithm for the considered problem}\label{AM_AG_fig1}
\end{figure}

The illustration of the algorithm for the problem is presented in Fig. \ref{AM_AG_fig1}. Bold segments indicate the optimal paths from the given point to $\Pi_0$.

\section{Existence of the solution and convergence of the algorithm}\label{AM_AG_sec4}

Let us discuss the existence of the solution. We are searching for a solution within the set $\mathbb{C}^1_L[0,l]$ of continuously differentiable functions on $[0,l]$, whose derivatives satisfy the Lipschitz condition with constant $L$. Denote
$$Y=\left\{f(x)\in \mathbb{C}^1_L[0,l]\colon f(0)=0,\ f(l)=y_l,\ |f'(0)|\leq L\right\}.$$
It can be proved that the family of functions $Y$ is uniformly bounded and equicontinuous. This implies that the conditions of the Arzela-Ascoli theorem are satisfied (see \cite{Lusternik_1975}), and therefore, the set $Y$ is relatively compact. It is also evident that $Y$ is closed, which implies that it is compact. 

\begin{theorem*}[Arzela-Ascoli]
A family of continuously differentiable functions $Y$ defined on $[0,l]$ is relatively compact in $\mathbb{C}^1[0,l]$ if and only if $Y$ is uniformly bounded and equicontinuous in $\mathbb{C}^1[0,l]$.
\end{theorem*}

On the other hand, it is obvious that the functional $J(y)$ depends continuously on $y$. Then, Weierstrass Extreme Value Theorem guarantees the existence of the minimum of the functional (\ref{AM_AG_J}) on $Y$. Thus, we arrive at the following result.

\begin{theorem*}[Existence of a solution]
There exists a solution of the boundary value problem (\ref{AM_AG_J})--(\ref{AM_AG_init_cond}) on the set $Y$.
\end{theorem*}

If we want to get a more accurate solution, we need to increase the density of the grid. This can be done in various ways. How should one decrease the grid step size in order to ensure the convergence of the corresponding sequence of piecewise linear approximations to the solution of the initial problem?
We give the answer in the main result of this section. But before that, we need to derive two auxiliary lemmas.

Let $\{Q_k\}_{k=1}^\infty$ be a sequence of grids satisfying the following properties:
\begin{itemize}
\item  For each $k$, the grid $Q_k$ has only one node at the points $x=0$ and $x=l$ with coordinates $y=0$ and $y=y_l$ respectively;
\item $\Delta_k\to 0$ and $\tau_k\to 0$ as $k\to\infty$, where $\Delta_k$ and $\tau_k$ are the discretization steps of grid $Q_k$ with respect to $y$ and $x$, respectively. 
\end{itemize} 

Then, the problem of finding the optimal piecewise linear trajectory $z(x,k)$ on the grid $Q_k$, in accordance with the previously established notations, can be expressed as follows:
\begin{equation}\label{AM_AG_J_discrete_k}
\begin{cases}
\min &J(y_0,\dots,y_{n_k})\\
\textit{s.t. } &y_i\in G_i^k,\ i=0,\dots,n_k,
\end{cases}
\end{equation} 
where $n_k=\frac{l}{\tau_k}$ and $Q_k=\displaystyle\bigcup_{i=0}^{n_k}\Pi_i^k$, with $\Pi_i^k=\left\{(x_i,y)\mid y\in G_i^k\right\}$, $i=0,\dots,n_k$ and $G_i^k=\left\{y_i^s(k)\mid s=0,\dots,N_k\right\}$ for $i=1,\dots,n_k-1$. Additionally, $N_k=\frac{y_l}{\Delta_k}$, $G_0^k=\left\{0\right\}$, $G_{n_k}^k=\left\{y_l\right\}$.

\begin{lemma}\label{AM_AG_lem1}
Let $y^\ast(x)$ be the optimal solution of the problem (\ref{AM_AG_J})--(\ref{AM_AG_init_cond}) and $z(x,k)$ be the optimal piecewise linear solution of the problem (\ref{AM_AG_J_discrete_k}), and $y(x,k)$ be the best fitting piecewise linear approximation of $y^\ast(x)$ on grid $Q_k$ in the sense that
\begin{equation}\label{AM_AG_lem_1_def_best_fit_y*}
y(x_i,k)=\underset{y\in G_i^k}{\operatorname{argmin}}|y_\ast(x_i)-y|,\ \forall i=0,\dots,n_k,
\end{equation} 
i.e., this piecewise linear trajectory $y(x,k)$ is composed of vertices that are closest to the corresponding values of the function $y^\ast(x)$.
If 
\begin{equation}\label{AM_AG_lem_1_cond1}
\lim_{k\to\infty}\left|J(y(x,k))-J(y^\ast(x))\right|=0,
\end{equation} 
then the sequence $\{z(x,k)\}_{k=1}^\infty$ converges to $y^\ast(x)$ in the sense of the functional $J$, i.e.
\begin{equation}\label{AM_AG_lem_1_cond2}
\lim_{k\to\infty}\left|J(z(x,k))-J(y^\ast(x))\right|=0.
\end{equation} 
\end{lemma}

\begin{proof}
Indeed, since for any $k$ we have 
\begin{equation}\label{AM_AG_lem_1_proof_1}J(y^\ast(x))\leq J(z(x,k))\leq J(y(x,k)),
\end{equation} 
condition (\ref{AM_AG_lem_1_cond1}) implies (\ref{AM_AG_lem_1_cond2}) .
\end{proof}

It must be noted that via the technique presented in \cite{AME_Abbasov_2024} can also be used to prove the uniqueness of the solution of the problem (\ref{AM_AG_J})--(\ref{AM_AG_init_cond}). This requires finding a variation of the functional and deriving the necessary condition for the minimum. Considering that, we state the following result under the assumption of uniqueness of the solution.
\begin{lemma}
If the solution of the problem (\ref{AM_AG_J})--(\ref{AM_AG_init_cond}) is unique, then the sequence $\{z(x,k)\}_{k=1}^\infty$ converges uniformly to $y^\ast(x)$, i.e.
\begin{equation}\label{AM_AG_lem_1_cor_cond1}
\lim_{k\to\infty}\left\|z(x,k))-y^\ast(x)\right\|_{\mathbb{C}[0,l]}=0,
\end{equation} 
where $$\left\|z(x,k))-y^\ast(x)\right\|_{C[0,l]}=\max_{x\in[0,l]}\left|z(x,k))-y^\ast(x)\right|$$ is the norm of Banach space $\mathbb{C}[0,l]$.
\end{lemma}
\begin{proof}
Assume the contrary. Since the solution $y^\ast(x)$ is unique for all $k$ we have
\begin{equation}\label{AM_AG_lem_1_cor_00}
J(z(x,k))>J(y^\ast(x))+a,
\end{equation} 
for some small $a>0$.

Let us show that
\begin{equation}\label{AM_AG_cor_eq0}
\left\|y(x,k)-y^\ast(x)\right\|_{\mathbb{C}[0,l]}\underset{k\to\infty}{\longrightarrow} 0,
\end{equation}
where $y(x,k)$ is the best fitting piecewise linear approximation of $y^\ast(x)$ on grid $Q_k$ as defined in (\ref{AM_AG_lem_1_def_best_fit_y*}).
We have 
$$y(i\tau_k,k))=y^\ast(i\tau_k)+O(\Delta_k), \quad \forall i=0,\dots,n_k,\quad \forall k\in\mathbb{N},$$
and for any $\bar{x}$ in $[0,l]$ and $k$ we can find $\bar{i}\in\{0,\dots,n_k-1\}$ such that $$\bar{i}\tau_k\leq \bar{x}\leq (\bar{i}+1)\tau_k.$$

Since the function $y^\ast(x)$ is continuous on the compact set $[0,l]$, it is uniformly continuous according to the Heine-Cantor theorem, i.e. for any $\varepsilon>0$ there exists $\delta>0$ such that for any $x'$ and $x''$ in $[0,l]$ satisfying the inequality $|x'-x''|\leq\delta$ the following condition holds: $|y^\ast(x')-y^\ast(x'')|\leq\varepsilon/2$.

We have $\tau_k\to 0$ and $\Delta_k\to 0$ as $k\to\infty$, therefore we can find $K$ such that for any $k>K$ the inequalities $\tau_k<\delta$ and $O(\Delta_k)<\varepsilon/2$ hold. For any $k>K$ we have
\begin{equation}\label{AM_AG_lem1_cor_1}
\begin{split}|y(\bar{x},k)-y^\ast(\bar{x})|\leq\max\{|y(\bar{i}\tau_k,k)-y^\ast(\bar{x})|,|y((\bar{i}+1)\tau_k,k)-y^\ast(\bar{x})|\}=\\
\max\{|y^\ast(\bar{i}\tau_k)-y^\ast(\bar{x})|,|y^\ast((\bar{i}+1)\tau_k)-y^\ast(\bar{x})|\}+|O(\Delta_k)|\leq\varepsilon.
\end{split}
\end{equation}
The inequality (\ref{AM_AG_lem1_cor_1}) was obtained for arbitrary $\bar{x}$ in $[0,l]$. Putting it all together, we get that for every $\varepsilon>0$ there exists $K$ such that for any $k>K$ the inequality
$$|y(\bar{x},k)-y^\ast(\bar{x})|\leq\varepsilon$$ is true for any $\bar{x}$ in $[0,l]$, which implies (\ref{AM_AG_cor_eq0}). Thus, piecewise linear approximations of the continuous function converge uniformly to the function as the number of approximations tends to infinity. Therefore we get
$$J(y(x,k))\underset{k\to\infty}{\longrightarrow} J(y^\ast(x)).$$
Then for sufficiently large $k$ we have 
$$J(y(x,k))\leq J(y^\ast(x))+\frac{a}{2}.$$
Then from (\ref{AM_AG_lem_1_proof_1}) and (\ref{AM_AG_lem_1_cor_00}) we get
$$J(y^\ast(x))+a< J(z(x,k))\leq J(y(x,k))<J(y^\ast(x))+\frac{a}{2},$$
which leads to the contradiction $a<\frac{a}{2}.$
\end{proof}

Therefore, according to Lemma \ref{AM_AG_lem1}, we must analyze the difference $$\Delta J_k=|J(y(x,k))-J(y^\ast(x))|$$ to prove the convergence of the dynamic programming scheme. We note that
\begin{equation}\label{AM_AG_pre_th_inequality}
\Delta J_k\leq\sum_{i=0}^{n_k-1}|J_i(y(i\tau_k,k),y((i+1)\tau_k,k))-J_i(y^\ast(i\tau_k,k),y^\ast((i+1)\tau_k,k))|.
\end{equation}

Now we can proceed to the main result.

\begin{theorem}\label{AM_AG_main_th}
Let the considered region $\Omega$, on which we solve the problem (\ref{AM_AG_J})--(\ref{AM_AG_init_cond}), be closed and bounded, and let there exist a twice continuously differentiable solution $y^\ast(x)\in \mathbb{C}^1(\Omega)$. If $\alpha(x,y)\in \mathbb{C}^1(\Omega)$, $\beta(x,y)\in \mathbb{C}^1(\Omega)$ and $\varphi(x,y)\in \mathbb{C}^2(\Omega)$, then for the proposed algorithm to converge to the solution of the problem it is sufficient that the grid steps satisfy the condition 
\begin{equation}\label{AM_AG_main_cond}
\Delta_k=\gamma\tau_k^{1+\varepsilon},
\end{equation} 
where $\gamma$, $\varepsilon$ are arbitrary positive constants and $k$ is the  iteration number.
\end{theorem}

\begin{proof}
For convenience, we omit the $k$ indexing in the proof and use the notation $y_i$ for $y(i\tau)$, $i=0,\dots,n-1$.
Using the left Riemann sum for the numerical integration and the estimate for the error of this formula, we obtain
\begin{equation*}
\begin{split}
J(&y)=\sum_{i=0}^{n-1}\alpha(i\tau,y_i)\sqrt{1+\left(\frac{y_{i+1}-y_i}{\tau}\right)^2+\left(\varphi_x(i\tau,y_i)+\varphi_y(i\tau,y_i)\left(\frac{y_{i+1}-y_i}{\tau}\right)\right)^2}\\& \sum_{k=0}^{i}\sqrt{1+\left(\frac{y_{k+1}-y_k}{\tau}\right)^2+\left(\varphi_x(k\tau,y_k)+\varphi_y(k\tau,y_k)\left(\frac{y_{k+1}-y_k}{\tau}\right)\right)^2}\tau^2+\\ &\sum_{i=0}^{n-1}\beta(i\tau,y_i)\sqrt{1+\left(\frac{y_{i+1}-y_i}{\tau}\right)^2+\left(\varphi_x(i\tau,y_i)+\varphi_y(i\tau,y_i)\left(\frac{y_{i+1}-y_i}{\tau}\right)\right)^2}\tau+O(\tau).
\end{split}
\end{equation*}

Denoting $$\Phi(y_i,y_{i+1})=\sqrt{1+\left(\frac{y_{i+1}-y_i}{\tau}\right)^2+\left(\varphi_x(i\tau,y_i)+\varphi_y(i\tau,y_i)\left(\frac{y_{i+1}-y_i}{\tau}\right)\right)^2},$$ 
we can rewrite expression for $J$:
\begin{equation*}
J(y)=\sum_{i=0}^{n-1} J_i(y_i,y_{i+1})+O(\tau)
\end{equation*}
where 
$$J_i(y_i,y_{i+1})=\alpha(i\tau,y_i)\Phi(y_i,y_{i+1}) \sum_{j=0}^{i}\Phi(y_j,y_{j+1})\tau^2+\beta(i\tau,y_i)\Phi(y_i,y_{i+1})\tau.$$

Let us consider the expansion of the function $\Phi(y_i,y_{i+1})$ in the neighborhood of the point $(y^\ast_i,y^\ast_{i+1})$, where $y^\ast_i=y^\ast(i\tau)$, $i=0,\dots,n$ is the value of the solution in the $i$-th node of the grid.
\begin{equation*}
\begin{split}
\Phi(y_i,y_{i+1})=\Phi^\ast_i+\Phi_{y_i}^{c}h_i+\Phi_{y_{i+1}}^{c}h_{i+1},
\end{split}
\end{equation*}
where $$\Phi^\ast_i=\Phi(y^\ast_i,y^\ast_{i+1}),$$ $$\Phi_{y_i}^{c}=\frac{\partial}{\partial y_i}\Phi(y_i+\theta(y^\ast_i-y_i),y_{i+1}+\theta(y^\ast_{i+1}-y_{i+1})),$$
$$\Phi_{y_{i+1}}^{c}=\frac{\partial}{\partial y_{i+1}}\Phi(y_i+\theta(y^\ast_{i+1}-y_{i+1}),y_{i+1}+\theta(y^\ast_{i+1}-y_{i+1})),$$
and $\theta\in[0,1]$, $h_i=y_i-y^\ast_i$, for all $i=0,\dots,n$. We have 
\begin{equation}\label{AM_GA_expansion_J_i_init}
\begin{split}
J(y_i,y_{i+1})=&\left(\alpha(\tau i,y^\ast_i)+\alpha_y(\tau i,y^{c_\alpha}_i)h_i\right)\left(\Phi^\ast_i+\Phi_{y_i}^{c}h_i+\Phi_{y_{i+1}}^{c}h_{i+1}\right)\\ &\left(\sum_{j=0}^{i}\left(\Phi^\ast_j+\Phi_{y_j}^{c}h_j+\Phi_{y_{j+1}}^{c}\Delta_{j+1}\right)\tau^2\right)+\\&\left(\Phi^\ast_i+\Phi_{y_i}^{c}h_i+\Phi_{y_{i+1}}^{c}h_{i+1}\right)\left(\beta(\tau i,y^\ast_i)+\beta_y(\tau i,y^{c_\beta}_i)h_i\right)\tau,
\end{split}
\end{equation}
where $y^{c_\alpha}_i=y_{i}+\theta_\alpha(y^\ast_i-y_{i})$, $y^{c_\beta}_i=y_{i}+\theta_\beta(y^\ast_i-y_{i})$ and $\theta_\alpha$, $\theta_\beta$ are in $[0,1]$.

From (\ref{AM_GA_expansion_J_i_init}) we derive
\begin{equation}\label{AM_GA_expansion_J_i}
\begin{split}
J(y_i,&y_{i+1})=J(y^\ast_i,y^\ast_{i+1})+\\&\alpha_y(\tau i,y^{c_\alpha}_i)h_i\left(\Phi^\ast_i+\Phi_{y_i}^{c}h_i+\Phi_{y_{i+1}}^{c}h_{i+1}\right) \sum_{j=0}^{i}\left(\Phi^\ast_j+\Phi_{y_j}^{c}h_j+\Phi_{y_{j+1}}^{c}h_{j+1}\right)\tau^2+\\ &\alpha(\tau i,y^{\ast}_i)\left(\Phi_{y_i}^{c}h_i+\Phi_{y_{i+1}}^{c}h_{i+1}\right)\sum_{j=0}^{i}\left(\Phi^\ast_j+\Phi_{y_j}^{c}h_j+\Phi_{y_{j+1}}^{c}h_{j+1}\right)\tau^2+\\ &\alpha(\tau i,y^{\ast}_i)\Phi^\ast_i\sum_{j=0}^{i}\left(\Phi^\ast_j+\Phi_{y_j}^{c}h_j+\Phi_{y_{j+1}}^{c}h_{j+1}\right)\tau^2+\\&\left(\Phi_{y_i}^{c}h_i+\Phi_{y_{i+1}}^{c}h_{i+1}\right)\left(\beta(\tau i,y^\ast_i)+\beta_y(\tau i,y^{c_\beta}_i)h_i\right)\tau+\Phi^\ast_i\beta_y(\tau i,y^{c_\beta}_i)h_i\tau.
\end{split}
\end{equation}

Since the functions $\varphi_x$, $\varphi_y$, $\alpha$, $\alpha_y$, $\beta$, $\beta_y$ are continuous on the closed and bounded set $\Omega$, they are bounded on the set. 

Obviously, we have 
\begin{equation}\label{AM_GA_expansion_Phi_i_}
\begin{split}
\Phi_{y_i}^{c}=&\frac{1}{\Phi(y_i^\ast+\theta h_i,y_{i+1}^\ast+\theta h_{i+1})}\Bigg[\frac{y_i^\ast-y_{i+1}^\ast+\theta(h_i-h_{i+1})}{\tau}+\\&\bigg(\varphi_{x}(i\tau,y_i^\ast+\theta h_i)-\varphi_{y}(i\tau,y_i^\ast+\theta h_i)\frac{y_i^\ast-y_{i+1}^\ast+\theta(h_i-h_{i+1})}{\tau}\bigg)\\&\bigg(\varphi_{xy}(i\tau,y_i^\ast+\theta h_i)-\varphi_{yy}(i\tau,y_i^\ast+\theta h_i)\frac{y_i^\ast-y_{i+1}^\ast+\theta(h_i-h_{i+1})}{\tau}\bigg)-\\&
\varphi_{y}(i\tau,y_i^\ast+\theta h_i)\frac{1}{\tau}\Bigg],
\end{split}
\end{equation}
\begin{equation}\label{AM_GA_expansion_Phi_i+1}
\begin{split}
\Phi_{y_{i+1}}^{c}=&\frac{1}{\Phi(y_i^\ast+\theta h_i,y_{i+1}^\ast+\theta h_{i+1})}\Bigg[\frac{y_{i+1}^\ast-y_i^\ast+\theta(h_{i+1}-h_i)}{\tau}+\\&\bigg(\varphi_{x}(i\tau,y_i^\ast+\theta h_i)-\varphi_{y}(i\tau,y_i^\ast+\theta h_i)\frac{y_i^\ast-y_{i+1}^\ast+\theta(h_i-h_{i+1})}{\tau}\bigg)\\&
\varphi_{y}(i\tau,y_i^\ast+\theta h_i)\frac{1}{\tau}\Bigg],
\end{split}
\end{equation}
whence, noting that $|h_i|\leq\Delta$ for all $i=0,\dots,n,$ where $\Delta$ is the step of the grid with respect to the variable $y$,  we come to the following inequalities which are valid under the conditions of the theorem 
\begin{equation}\label{AM_GA_expansion_Phi_i}
\begin{cases}
\Phi^\ast_i\leq c_0,\\
\Phi_{y_i}^{c}h_i+\Phi_{y_{i+1}}^{c}h_{i+1}\leq c_1\Delta+c_2\frac{\Delta}{\tau}+c_3\frac{\Delta^2}{\tau}+c_4\frac{\Delta^2}{\tau^2}+c_5\frac{\Delta^2}{\tau^3},\\
i=1,\dots,n-1,
\end{cases}
\end{equation}
where $c_0,\dots,c_5$ are constants.

Inequality  (\ref{AM_AG_pre_th_inequality}) can be rewritten as
$$\Delta J\leq\sum_{i=0}^{n-1}|J_i(y_i,y_{i+1})-J_i(y_i^\ast,y_{i+1}^\ast)|+O(\tau).
$$
So, using (\ref{AM_GA_expansion_J_i}) and (\ref{AM_GA_expansion_Phi_i_}), (\ref{AM_GA_expansion_Phi_i+1}), noting that $n\tau=1$, we conclude that for $\Delta J$ to tend to zero as the density of the grid increases, it is sufficient that $\Delta=\gamma\tau^{1+\varepsilon}$ for some positive constants $\gamma$ and $\varepsilon$, i.e. $\Delta$ tends to zero faster than $\tau$.
\end{proof}

\begin{remark}\label{AM_AG_main_th_cor}
From the inequalities (\ref{AM_GA_expansion_Phi_i}) we see that the upper estimates for the expressions $\left|\Phi_{y_i}^{c}\Delta_i+\Phi_{y_{i+1}}^{c}\Delta_{i+1}\right|$ for all $i=1,\dots,n$ become smaller for the higher values of $\varepsilon$, while the convergence is not guaranteed for $\varepsilon=0$.
\end{remark}

Let us get an estimate for the number of arithmetic operations $T$ needed to find the solution by the proposed method. Denote 
by $N=y_l/\Delta$ the number of nodes with respect to the variable $y$. Assuming that each calculation using the formula (\ref{AM_AG_d_kj}) requires $p-1$ arithmetic operations, we can write 
\begin{equation}\label{AM_GA_estimate_arithm_op_init_alg}
T\leq\sum_{i=0}^{n-1}pN^2=pN^2n.
\end{equation} 
Noticing that $n=l/\tau$ and using the Theorem \ref{AM_AG_main_th}, we continue the chain (\ref{AM_GA_estimate_arithm_op_init_alg}) as
\begin{equation}\label{AM_GA_estimate_arithm_op_init_alg_tau}
T\leq\frac{p}{\gamma^2} \frac{1}{\tau^{3+2\varepsilon}}=O\left(\frac{1}{\tau^{3+2\varepsilon}}\right).
\end{equation} 
\\
\\

At the end of this section, let us consolidate all the reasoning and present the dynamic programming algorithm for the problem (\ref{AM_AG_J})--(\ref{AM_AG_init_cond}) (see Algorithm \ref{AM_DP_alg_init}). Let $\tau$, $\Delta$ be some positive numbers and $n=l/\tau$, $N=y_l/\Delta$. Denote the grid by $Q$, i.e. 
$$Q=\left\{(x_i,y_j)\mid i\in I,\ j\in W_i\right\},$$
$I=\{0,\dots,n\}$, $W_0=\{0\}$, $W_{n}=\{N\}$, $W_i=\{0,\dots,N\}$ for $i$ in $\{1,\dots,n-1\}$.

\begin{algorithm}[H]
\caption{Dynamic programming algorithm for the problem (\ref{AM_AG_J})--(\ref{AM_AG_init_cond})}\label{AM_DP_alg_init}
\begin{algorithmic}[1]
\Require The grid $Q$ and the functions $J_i$ defined in (\ref{AM_AG_J_i}).
\Ensure  Piecewise-linear minimizer $z_\ast(x)$ of the functional $J$ on the grid $Q$.
\State $d_0^0=0$%$ \Leftarrow 1$
\For{$i\in I\setminus\{n\}$} 
   %\State 
         \For{$j_1\in W_{i+1}$} 
               %\State 
                       \For{$j_2\in W_{i}$} 
                             \State $\displaystyle d_i^{j_1j_2}=J_i(y_i^{j_2},y_{i+1}^{j_1})$
                        \EndFor
                        
           $d_{i+1}^{j_1}=\displaystyle\min_{s\in W_i}\left(d_i^s+d_i^{sj_1}\right)$
           
          \EndFor
\EndFor
\State Take the nodes of the grid $Q$ at which the distance $d_n^m$ is attained as the vertices of the piecewise linear function defining the solution $z_\ast(x)$.
\end{algorithmic}
\end{algorithm}

\begin{remark}\label{AM_AG_Alg1_Rem} Note that, via Algorithm \ref{AM_DP_alg_init} we can construct a sequence of the piecewise-linear approximations $\{z_\ast^k(x)\}_{k=1}^\infty$ that converges to the optimal solution of the problem (\ref{AM_AG_J})--(\ref{AM_AG_init_cond}). In order to do so, according to Theorem \ref{AM_AG_main_th}, we need to define a sequence $\{\tau_k\}$ of positive numbers tending to zero, take $\Delta_k=\tau_k^{1+\varepsilon}$ for some $\varepsilon>0$ and then by means of Algorithm \ref{AM_DP_alg_init} for any natural $k$ construct $z_\ast^k(x)$ on the grid $Q_k$ which is defined by step sizes $\tau_k$, $\Delta_k$.
\end{remark}

\section{Local search modification}\label{AM_AG_sec5}
This section outlines several options that make it faster to reach a solution using the described method.

\subsection{Modification for the local search}\label{AM_AG_sec5_1}
Algorithm \ref{AM_DP_alg_init} enables the finding of a global solution to the problem, but requires a high computational expense. It is possible to achieve a significant reduction in the computational costs spent on obtaining a solution $z_\ast(x)$ if a local solution is sought instead of a global one.

Let $\Gamma_k$ be a broken line which approximates the solution of the problem (\ref{AM_AG_J})--(\ref{AM_AG_init_cond}) and defined by nodes $\left(x_i,y_{j_i(k)}\right)$,  where $y_{j_i(k)}\in G_i$ for any $i=0,\dots,n$, i.e.
$$\Gamma_k=\left\{\left(x_i,y_{j_i(k)}\right)\mid i\in I,\ j_i(k)\in W_i\right\},$$
where $I=\{0,\dots,n\}$, $W_0=\{0\}$, $W_{n}=\{N\}$, $W_i=\{0,\dots,N\}$ for $i$ in $\{1,\dots,n-1\}$.

To build $\Gamma_{k+1}$ in each $\Pi_i$ we consider only $m$ nearest points to the node $\left(x_i,y_{j_i(k)}\right)$, $i=0,\dots,n$ and run previously stated algorithm on the truncated grid. Obviously, we get a new broken line $\Gamma_{k+1}$ in the neighborhood of $\Gamma_k$. The proposed evolution of the approximation can bring us to the piecewise-linear approximation of a local solution of the problem (\ref{AM_AG_J})--(\ref{AM_AG_init_cond}) (see Algorithm \ref{AM_DP_alg_local_search_modification}). 

\begin{figure}[h]
\centering
\includegraphics[width=0.7\textwidth]{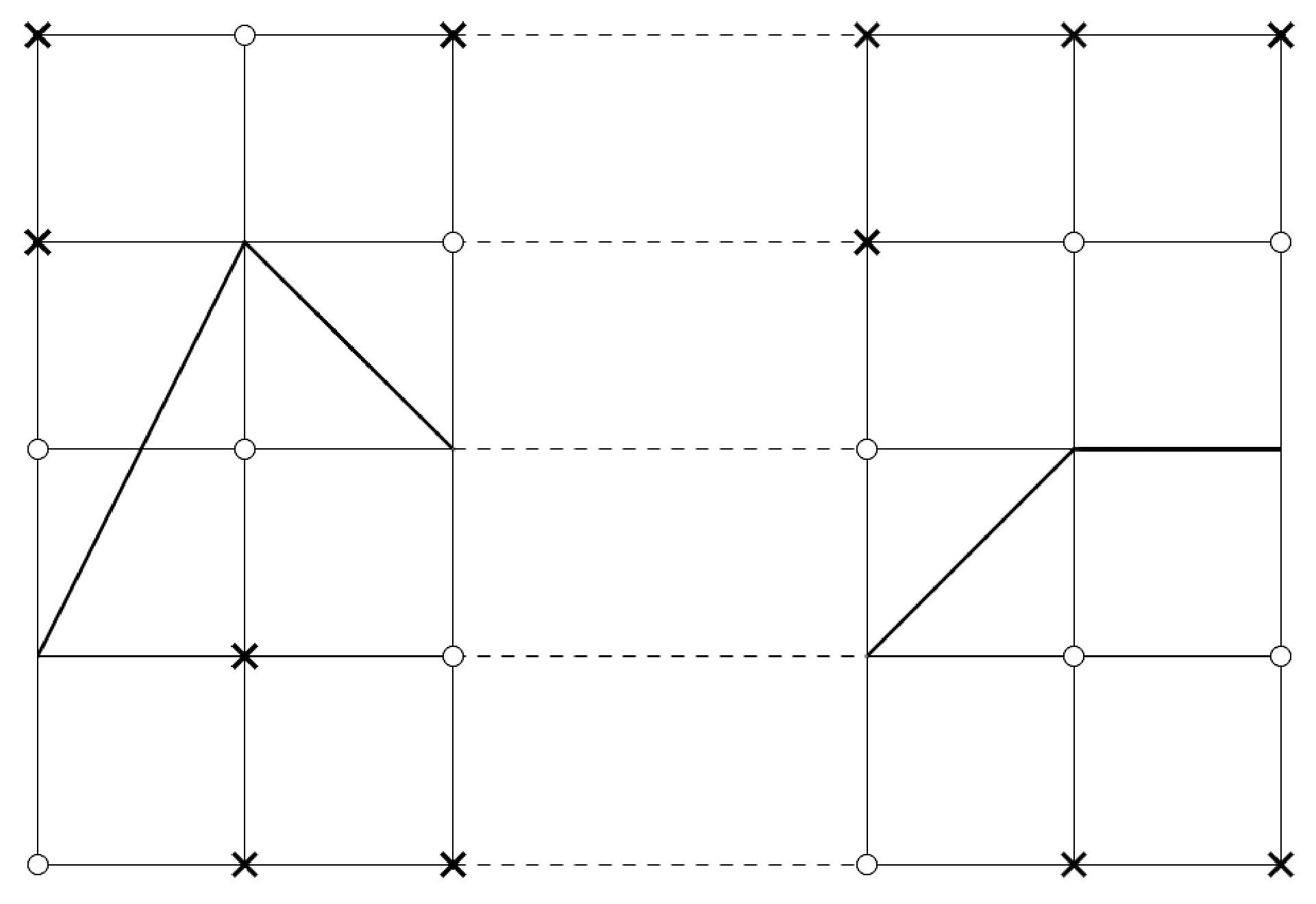}
\caption{An illustration of the algorithm for finding a local solution. Circles indicate feasible points at this iteration, crosses indicate infeasible points}\label{AM_AG_fig2}
\end{figure}
In Fig. \ref{AM_AG_fig2} we have $m=2$ and nodes of the truncated grid are represented by circles.

Applying the same reasoning as we used in the end of previous section we can get an estimate for the number of arithmetic operations  $T_j$ which are needed on  $j$-th iteration as
\begin{equation}\label{AM_GA_estimate_arithm_op_modif_alg}
T_j\leq p m^2n.
\end{equation} 
The total number of iterations  $T$ is less or equal to $\tilde{p}N/m$, where $\tilde{p}$ is a constant. This fact is acknowledged by numerical experiments and is quite intuitive because the bigger $m$ the less is the number of total iterations (see \cite{AME_ Moiseev_1971}). If $m=N$, we have only one iteration, as our modification coincides with the original algorithm. Therefore, the total number of arithmetic operations here is
\begin{equation}\label{AM_GA_estimate_arithm_op_modif_alg_tau}
T\leq \tilde{p}pmNn=O\left(\frac{1}{\tau^{2+\varepsilon}}\right),
\end{equation}
which is less than (\ref{AM_GA_estimate_arithm_op_init_alg_tau}).\\

Let us state the modification of dynamic programming algorithm for the problem (\ref{AM_AG_J})--(\ref{AM_AG_init_cond}) in the same manner as we did for its initial variant. Let $\tau$, $\Delta$ be some positive numbers  and $n=l/\tau$, $N=y_l/\Delta$. Denote the grid by $Q$, i.e. 
$$Q=\left\{(x_i,y_j)\mid i\in I,\ j\in W_i\right\},$$
$I=\{0,\dots,n\}$, $W_0=\{0\}$, $W_{n}=\{N\}$, $J_i=\{0,\dots,N\}$ for $i$ in $\{1,\dots,n-1\}$. Let also $z_k(x)$ be a piecewise-linear approximation of the solution on the grid $Q$ on $k$-th iteration and $\Gamma_k$ the corresponding broken line, i.e. $z_k(x_i)=y_{j_i(k)}$. 

\begin{algorithm}[H]
\caption{Modified dynamic programming algorithm for the problem (\ref{AM_AG_J})--(\ref{AM_AG_init_cond})}\label{AM_DP_alg_local_search_modification}
\begin{algorithmic}[1]
\Require Natural number $m$ and initial piecewise-linear approximation $z_0(x)$ of the solution on the grid $Q$, that is, the piecewise linear function defined by $z_0(x)$ has only the vertices $\{(x_i,z_0(x_i))\mid i\in I\}$, and all of them are contained within $Q$.
\Ensure Piecewise-linear minimizer $z_\ast(x)$ of the functional $J$ on the grid $Q$
\State $k \Leftarrow 0$
\Do 
\State Construct sets $W_i(k)=\left\{j_i(k)-m,\dots,j_i(k),\dots,j_i(k)+m\right\}$ for any $i\in I$, where all indices less than zero or higher than $N$ should be omitted.
\State Define grid $Q_k=\left\{(x_i,y_j)\mid j\in W_i(k)\right\}$
\State Run Algorithm \ref{AM_DP_alg_init} on $Q_k$ and get $z_{k+1}(x)$ as a result
\State $k \Leftarrow k+1$
\doWhile{ there exists $i\in I$ such that $z_{k}(x_i)\neq z_{k-1}(x_i)$}
\end{algorithmic}
\end{algorithm} 

\begin{remark}\label{AM_AG_Alg2_Rem} Algorithm \ref{AM_DP_alg_local_search_modification} can be built in each iteration of the algorithm described in Remark \ref{AM_AG_Alg1_Rem}, and in this case, it generally leads us to a local solution of the problem (\ref{AM_AG_J})--(\ref{AM_AG_init_cond}).
\end{remark}

\subsection{2D simplification of the model}\label{AM_AG_sec5_2}
Another potential way to reduce the computational complexity of the proposed solution is in cases where the height difference over the terrain under consideration is insignificant. In this case $z'(x)$ is negligible and the cost functional (\ref{AM_AG_J}) can be rewritten as
\begin{equation}\label{AM_AG_J2D}
\begin{split}
J(y)=\int_{0}^{l}\alpha(x,y(x))\sqrt{1+{y'}^2(x)}\int_{0}^{x} &\sqrt{1+{y'}^2(\xi)}d\xi dx+\\
&\int_{0}^{l}\beta(x,y(x))\sqrt{1+{y'}^2(x)}dx.
\end{split}
\end{equation}
More detail regarding this problem can be found in  \cite{AME_Abbasov_2021,AME_Abbasov_2024}.

Discretizing (\ref{AM_AG_J2D}) we get
\begin{equation}\label{AM_AG_J_i_2D}
\begin{split}
J_i(y_i,y_{i+1})=\int_{x_i}^{x_{i+1}}\alpha(x,y_i(x))\sqrt{1+{y'_i}^2}\int_{0}^{x}& \sqrt{1+{y'_i}^2}d\xi dx+\\
&\int_{x_i}^{x_{i+1}}\beta(x,y_i(x))\sqrt{1+{y'_i}^2}dx.
\end{split}
\end{equation}

\section{Numerical experiments}\label{AM_AG_sec5}
Let us apply the proposed algorithms to 2D and 3D problems to demonstrate their results. We will use the results of the Ritz method as a benchmark for comparison. The system of basis functions for the Ritz method is $\left\{\sin{\frac{\pi k x}{l}}\right\}_{k=1}^\infty$. 

\begin{example} \label{AM_GA_exmpl1}
Consider the problem of minimizing (\ref{AM_AG_J2D}) for $$\alpha(x,y)=\cos^2{5x}\cos^2{y},$$ 
$$\beta(x,y)=1+\sin{5x}\sin{y},$$
and boundary conditions $$y(0)=0, \ y(l)=1,$$
where $l=1$.

The results of Algorithm \ref{AM_DP_alg_init} for different values of the parameters $\tau$ and $\varepsilon$ are presented in Table \ref{AM_GA_tab1_1}. Here, $\Delta=\tau^{1+\varepsilon}$ is the step size of the grid with respect to the variable $y$, as derived in Theorem \ref{AM_AG_main_th}. Increasing the parameter $\varepsilon$ guarantees an increase in computation time.

\begin{table}[h]
\caption{Results of numerical experiments in Example \ref{AM_GA_exmpl1} for Algorithm \ref{AM_DP_alg_init}}\label{AM_GA_tab1_1}
\begin{tabular*}{\textwidth}{@{\extracolsep\fill}lcccccccc}
\toprule%
& \multicolumn{2}{@{}c@{}}{$\varepsilon=0$} & \multicolumn{2}{@{}c@{}}{$\varepsilon=0.25$}& \multicolumn{2}{@{}c@{}}{$\varepsilon=0.5$}& \multicolumn{2}{@{}c@{}}{$\varepsilon=0.75$}  \\\cmidrule{2-3}\cmidrule{4-5}\cmidrule{6-7}\cmidrule{8-9}%
$\tau$ & $J$ & time\footnotemark[1] & $J$  &   time\footnotemark[1]  & $J$ &  time\footnotemark[1]& $J$  &   time\footnotemark[1]     \\
\midrule
$1/4$   & 1.54301 & 0.01  & 1.49633  & 0.01  & 1.48259 & 0.02   & 1.48152 & 0.04\\
$1/8$   & 1.51742 & 0.05  & 1.45310  &  0.14   & 1.44712 & 0.36   & 1.44254 & 0.92\\
$1/16$ & 1.51455 & 0.40  & 1.44337  & 1.56 &  1.44010 & 6.02& 1.43899 & 23.52\\
$1/32$ & 1.51129 & 3.31 & 1.43833 & 18.42& 1.43368 & 100.6& 1.43327 & 535.7\\
\bottomrule
\end{tabular*}
\footnotetext[1]{Time of computations in seconds.}
\end{table}

Now, we proceed to Algorithm \ref{AM_DP_alg_local_search_modification}. Our experiments have shown that the choice of the parameter $m$ for this algorithm has no effect on the final results compared to the parameters $\tau$ and $\varepsilon$. Higher values of $m$ only increase the computational complexity and do not influence the obtained trajectory cost. Therefore, we perform our computations only for $m=1$. The results are presented in Table \ref{AM_GA_tab1_2}.

\begin{table}[h]
\caption{Results of numerical experiments in Example \ref{AM_GA_exmpl1} for Algorithm \ref{AM_DP_alg_local_search_modification} with $m=1$}\label{AM_GA_tab1_2}
\begin{tabular*}{\textwidth}{@{\extracolsep\fill}lcccccc}
\toprule%
& \multicolumn{3}{@{}c@{}}{$\varepsilon=0.5$} & \multicolumn{3}{@{}c@{}}{$\varepsilon=0.75$} \\\cmidrule{2-4}\cmidrule{5-7}%
$\tau$ & $J$ & iter\footnotemark[1] & time\footnotemark[2] & $J$  &iter\footnotemark[1] &   time\footnotemark[2]  \\
\midrule
$1/16$   & 1.44010 & 45  & 0.58  & 1.43899   & 89  & 1.32  \\
$1/32$   & 1.43368 & 175  & 5.33 &  1.43370   & 414  & 14.9   \\
$1/64$ & 1.43308 & 509  & 37.9  & 1.43247 & 1433 & 166.47\\
\bottomrule
\end{tabular*}
\footnotetext[1]{Number of iterations.}
\footnotetext[2]{Time of computations in seconds.}
\end{table}

We see that the larger $\varepsilon$, the better Algorithm \ref{AM_DP_alg_local_search_modification} performs, especially in comparison to Algorithm \ref{AM_DP_alg_init}, which corresponds to Remark \ref{AM_AG_main_th_cor} and the estimates (\ref{AM_GA_estimate_arithm_op_init_alg_tau}) and (\ref{AM_GA_estimate_arithm_op_modif_alg_tau}).

The Ritz method produces a solution of the form 
$$y(x)=x+\sum_{k=1}^{10}a_k\sin{\frac{\pi k x}{l}},$$
where 
\begin{equation*}
\begin{split}
&a_1=-0.342929,\   a_2=0.132031,\   a_3=-0.083452,\   a_4=0.046821  ,\   a_5=-0.027623,\\   
&a_6=0.017216,\   a_7=-0.009948,\   a_8=0.005366,\   a_9=-0.002641,\   a_{10}=0.001030,
\end{split}
\end{equation*}
with a cost of $1.43743$ in $60.5$ seconds. Algorithm \ref{AM_DP_alg_local_search_modification} produces a solution of comparable quality in $1.32$ seconds (see Table \ref{AM_GA_tab1_2} for $\tau=1/16$, $\varepsilon=0.75$).

The graph of the optimal curve is shown in Figure \ref{AM_AG_fig_exmpl_2D}.

\begin{figure}[h]
\centering
\includegraphics[width=0.6\textwidth]{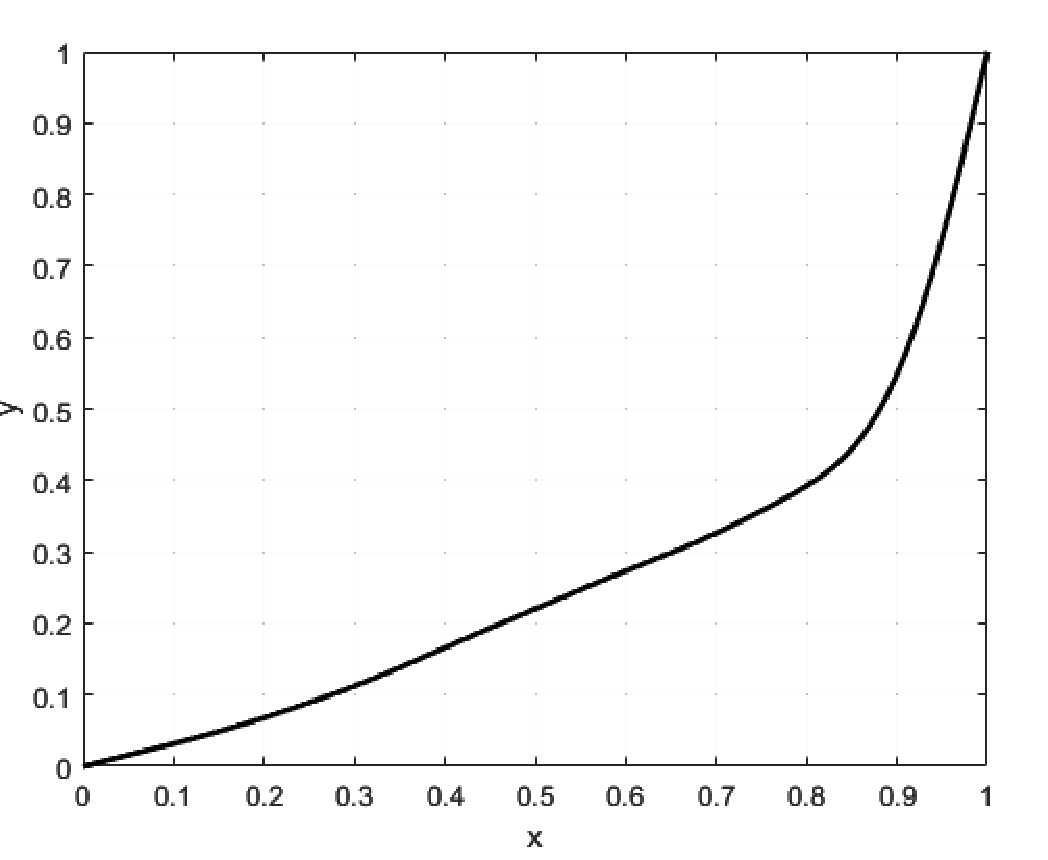}
\caption{A graph of the optimal solution in Example \ref{AM_GA_exmpl1}}\label{AM_AG_fig_exmpl_2D}
\end{figure}

\end{example}

\begin{example} \label{AM_GA_exmpl2}
Consider the problem of minimizing (\ref{AM_AG_J}) for $$\varphi(x,y)=\sin{5x}\sin{y},$$ 
$$\alpha(x,y)=0.1, \quad \beta(x,y)=0.5,$$
and boundary conditions $$y(0)=0, \ y(l)=1,$$
where $l=1$.

\begin{table}[h]
\caption{Results of numerical experiments in Example \ref{AM_GA_exmpl2} for Algorithm \ref{AM_DP_alg_init}}\label{AM_GA_tab2_1}
\begin{tabular*}{\textwidth}{@{\extracolsep\fill}lcccccccc}
\toprule%
& \multicolumn{2}{@{}c@{}}{$\varepsilon=0$} & \multicolumn{2}{@{}c@{}}{$\varepsilon=0.25$}& \multicolumn{2}{@{}c@{}}{$\varepsilon=0.5$}& \multicolumn{2}{@{}c@{}}{$\varepsilon=0.75$}  \\\cmidrule{2-3}\cmidrule{4-5}\cmidrule{6-7}\cmidrule{8-9}%
$\tau$ & $J$ & time\footnotemark[1] & $J$  &   time\footnotemark[1]  & $J$ &  time\footnotemark[1]& $J$  &   time\footnotemark[1]     \\
\midrule
$1/4$   & 1.17293 & 0.1  & 1.17506  & 0.2   & 1.15674 & 0.4   & 1.15300 & 0.9\\
$1/8$   & 1.17157 & 0.9  & 1.15596  & 2.8   & 1.14476 & 7.3   & 1.14032 & 20.2\\
$1/16$ & 1.16689 & 8.2  & 1.14852  & 32.7 & 1.13964 & 127.8& 1.13866 & 502.9\\
$1/32$ & 1.16406 & 67.4 & 1.14500 & 387.6& 1.13791 & 2156.7& 1.13750 & 11579.9\\
\bottomrule
\end{tabular*}
\footnotetext[1]{Time of computations in seconds.}
\end{table}

\begin{table}[h]
\caption{Results of numerical experiments in Example \ref{AM_GA_exmpl2} for Algorithm \ref{AM_DP_alg_local_search_modification} with $m=1$}\label{AM_GA_tab2_2}
\begin{tabular*}{\textwidth}{@{\extracolsep\fill}lcccccc}
\toprule%
& \multicolumn{3}{@{}c@{}}{$\varepsilon=0.5$} & \multicolumn{3}{@{}c@{}}{$\varepsilon=0.75$} \\\cmidrule{2-4}\cmidrule{5-7}%
$\tau$ & $J$ & iter\footnotemark[1] & time\footnotemark[2] & $J$  &iter\footnotemark[1] &   time\footnotemark[2]  \\
\midrule
$1/16$   & 1.13964 & 55  & 16.4  & 1.13866   & 108  & 32.5  \\
$1/32$   & 1.13791 & 164  & 101.7 &  1.13750   & 390  & 243.8   \\
$1/64$ & 1.13763 & 486  & 615.5  & 1.13719 & 1372 & 1873.7\\
\bottomrule
\end{tabular*}
\footnotetext[1]{Number of iterations.}
\footnotetext[2]{Time of computations in seconds.}
\end{table}

The Ritz method produces a solution of the form 
$$y(x)=x+\sum_{k=1}^{10}a_k\sin{\frac{\pi k x}{l}},$$
where 
\begin{equation*}
\begin{split}
&a_1=-0.370262,\   a_2=0.055788,\   a_3=0.010580,\   a_4=-0.008663  ,\   a_5=0.002984,\\   
&a_6=0.002658,\   a_7=-0.004272,\   a_8=0.004357,\   a_9=-0.002889,\   a_{10}=0.001418,
\end{split}
\end{equation*}
with a cost of $1.13763$ in $72.8$ seconds. 
The solution with a cost of $1.13711$, which is comparable to the best result of the local search modification of the algorithm (see Table \ref{AM_GA_tab2_2} for $\tau=1/64$, $\varepsilon=0.75$), can be obtained by the Ritz method for expansion with 30 terms in $2261.7$ seconds. 

\begin{figure}[h]
\centering
\includegraphics[width=0.85\textwidth]{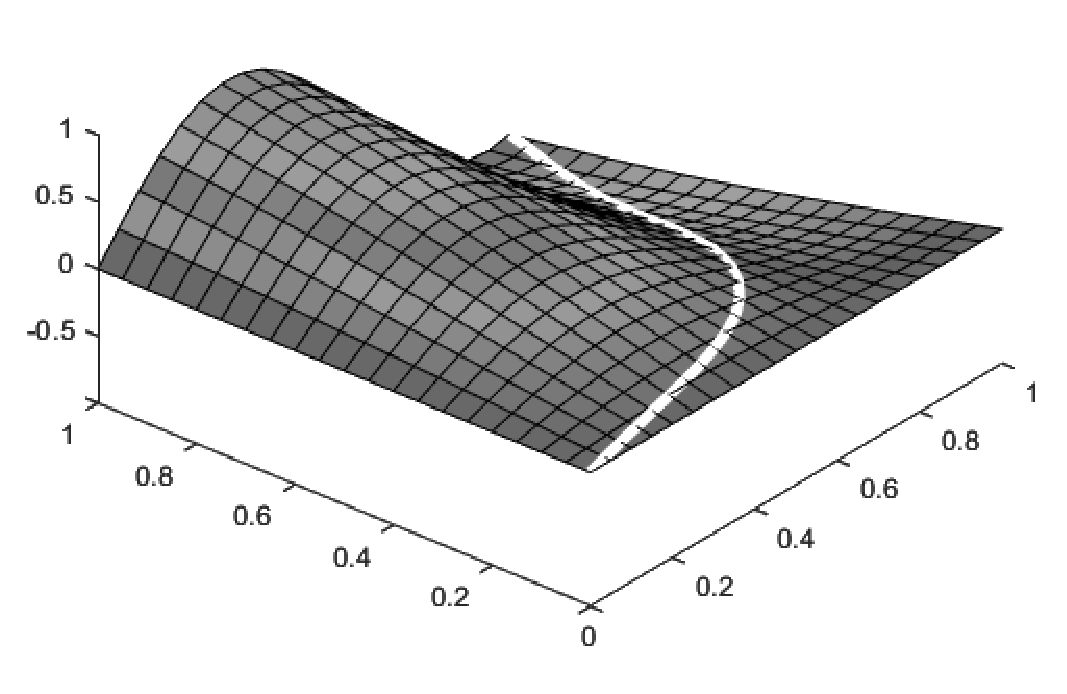}
\caption{A graph of the optimal solution in Example \ref{AM_GA_exmpl2}}\label{AM_AG_fig_exmpl_3D}
\end{figure}

The graph of the optimal curve is shown in Figure \ref{AM_AG_fig_exmpl_3D}. Since we are dealing with constants $\alpha$ and $\beta$ in this example, the length of the curve is the main factor to be minimized. Therefore, our solution expectedly avoids areas of high altitude.

\end{example}

\section{Conclusion}\label{sec13}

We consider the problem of obtaining the optimal trajectory in terms of construction cost that connects two points on the terrain. This problem is formulated as a calculus of variations problem, and we discuss the existence of a solution. We propose Algorithm \ref{AM_DP_alg_init}, which utilizes dynamic programming, and we prove its convergence in Theorem \ref{AM_AG_main_th}. Additionally, we analyze the computational complexity of the method, which is comparable to that of Dijkstra's algorithm due to the specific structure of the employed graph. While Algorithm \ref{AM_DP_alg_init} allows for the determination of a global solution to the problem, the number of required arithmetic operations increases rapidly as we enhance the required accuracy of the solution. Therefore, we also construct Algorithm \ref{AM_DP_alg_local_search_modification}, a modification of the initial algorithm that requires less computational effort, though this comes at the cost of generally providing only a local solution. Nevertheless, our numerical experiments demonstrate that Algorithm \ref{AM_DP_alg_local_search_modification} produces solutions of comparable quality to those obtained with Algorithm \ref{AM_DP_alg_init}.

The results obtained enable the identification of a suitable law for increasing grid density, facilitating the achievement of a solution with a specified level of accuracy. In contrast to the widely used Dijkstra's algorithm in this field, the method developed and analyzed in this study offers a relatively straightforward computational approach to obtaining a solution to the original problem with the desired accuracy.

It is important to note that the methods described in this paper do not undergo significant changes when applied to problems with constraints. For such problems, it is necessary to exclude points corresponding to "forbidden" regions from the sets $G_i$, where $i=0,\dots,n$.

\end{document}